\documentclass[a4paper,10pt,reqno]{amsart}
\usepackage[utf8]{inputenc}
\usepackage{amsmath}
\usepackage{cases}
\usepackage{amsfonts}
\usepackage[colorlinks,linkcolor=blue,citecolor=blue]{hyperref}
\usepackage{latexsym, amssymb, amsmath, amsthm, bbm}
\usepackage[all]{xy}
\usepackage{pgfplots}
\usepackage{enumerate}
\usepackage{mathrsfs}

\DeclareSymbolFont{EulerExtension}{U}{euex}{m}{n}
\DeclareMathSymbol{\euintop}{\mathop} {EulerExtension}{"52}
\DeclareMathSymbol{\euointop}{\mathop} {EulerExtension}{"48}

\allowdisplaybreaks[4]

\def \id{\operatorname{id}}

\def \k{\Bbbk}

\usepackage{mathtools}
\usepackage{amssymb}
\numberwithin{equation}{section}

\newtheorem{theorem}{Theorem}[section]
\newtheorem{lemma}[theorem]{Lemma}
\newtheorem{proposition}[theorem]{Proposition}
\newtheorem{corollary}[theorem]{Corollary}
\newtheorem{definition}[theorem]{Definition}
\newtheorem{example}[theorem]{Example}
\newtheorem{remark}[theorem]{Remark}

\begin{document}
\title[Derived categories of module categories]{On derived categories of module categories over multiring categories}
\thanks{}
\author[J. Yu]{Jing Yu}
\address{School of Mathematical Sciences, University of Science and Technology of China, Hefei 230026, People's Republic
of China}
\email{yujing46@ustc.edu.cn}
\maketitle

\date{}
\begin{abstract}
Let $\mathcal{A}$ and $\mathcal{B}$ be subcategories of tensor categories $\mathcal{C}$ and $\mathcal{D}$, respectively, both of which are abelian categories with finitely many isomorphism classes of simple objects. We prove that if their derived categories $\mathbf{D}^b(\mathcal{A})$ and $\mathbf{D}^b(\mathcal{B})$ are left triangulated tensor ideals and are equivalent as triangulated $\mathbf{D}^b(\mathcal{C})$-module categories via an equivalence induced by a monoidal triangulated functor $F:\mathbf{D}^b(\mathcal{C})\rightarrow \mathbf{D}^b(\mathcal{D})$, then the original module categories $\mathcal{A}$ and $\mathcal{B}$ are themselves equivalent. We then apply this result to smash product algebras. Furthermore, the localization theory of module categories and triangulated module categories is investigated.
\end{abstract}
\maketitle
\section{Introduction}
Derived categories and their equivalences, introduced by Grothendieck and Verdier \cite{Ver77}, play a vital role in the representation theory of finite-dimensional algebras and finite groups (see \cite{CR08, Hap88, Rou06}). In general, the equivalence of derived categories of abelian categories does not imply the equivalence of the abelian categories themselves. However, if such a derived equivalence carries additional structure, it becomes possible for the underlying abelian categories to be equivalent. In algebraic geometry, the term reconstruction theorems refers to results that characterize this phenomenon, as shown in \cite{Bal02,BO01}. In the case of algebras, Aihara and Mizuno \cite{AM17} proved that a preprojective algebra of Dynkin type is derived equivalent only to itself up to Morita equivalence. Zhang and Zhou \cite{ZZ22} showed that an analogous phenomenon holds for finite-dimensional hereditary weak bialgebras. Furthermore, Xu and Zheng \cite{XZ25} demonstrated that for tensor categories with finitely many isomorphism classes of simple objects, tensor equivalence corresponds precisely to an equivalence between their derived categories as monoidal triangulated categories.

There is no doubt that representations of weak bialgebras and tensor categories share the common feature of possessing a tensor structure. A natural question then arises: if we no longer consider tensor structure but instead introduce an additional module action induced by a monoidal triangulated functor, can derived equivalence and Morita equivalence imply each other?

More specifically, we aim to prove the following theorem, which appears as Theorem \ref{thm:eq=} in this paper:
\begin{theorem}\label{thm:0}
Let $\mathcal{C}$ and $\mathcal{D}$ be tensor categories, and let subcategories $\mathcal{A} \subset \mathcal{C}$ and $\mathcal{B} \subset \mathcal{D}$ be abelian categories with finitely many isomorphism classes of simple objects. Assume further that $\mathbf{D}^b(\mathcal{A})$ and $\mathbf{D}^b(\mathcal{B})$ are left triangulated tensor ideals of $\mathbf{D}^b(\mathcal{C})$ and $\mathbf{D}^b(\mathcal{D})$, respectively.
Suppose $F: \mathbf{D}^b(\mathcal{C}) \to \mathbf{D}^b(\mathcal{D})$ is a monoidal triangulated functor such that $F(\mathbf{D}^b(\mathcal{A})) \subseteq \mathbf{D}^b(\mathcal{B})$ and that its restriction $F|_{\mathbf{D}^b(\mathcal{A})}: \mathbf{D}^b(\mathcal{A}) \to \mathbf{D}^b(\mathcal{B})$ is an equivalence. Then $\mathcal{A}$ and $\mathcal{B}$ are equivalent as left $\mathcal{C}$-module categories.
\end{theorem}

In fact, Theorem \ref{thm:0} can be viewed as a generalization of the result by Xu and Zheng, which corresponds to the case where the subcategories $\mathcal{A}$ and $\mathcal{B}$ are taken to be the entire tensor categories themselves. It should be noted that the validity of Xu and Zheng's conclusion for tensor categories with infinitely many isomorphism classes of simple objects remains unknown. Theorem \ref{thm:0} can be applied in studying the local structure of monoidal triangulated equivalences between derived categories of tensor categories with infinitely many isomorphism classes of simple objects (such as the representation categories of certain infinite-dimensional Hopf algebras analogous to quantum groups). More generally, Theorem \ref{thm:0} extends to other contexts, including the module categories of algebras in a tensor category, such as the module categories of smash product algebras (see Example \ref{ex:crossed}).

This paper builds upon fundamental theories of module categories and triangulated module categories. In fact, the theory of module categories over monoidal categories was first systematically developed by Ostrik \cite{Ost03}, and was later applied, by Ostrik together with Etingof, Gelaki and Nikshych, to the study of tensor categories \cite{ENO05, EGNO15}. Meanwhile, Stevenson \cite{Ste13} introduced the concept of a triangulated module category to study the thick submodules of a compactly generated triangulated category $\mathcal{K}$ acted upon by a monoidal triangulated category $\mathcal{T}$. Stevenson's work serves as a relative counterpart to Balmer's tensor triangular geometry \cite{Bal10}.

To prove Theorem \ref{thm:0}, we employ compatible t-structures and triangulated module t-structures. We establish a general K\"{u}nneth formula on triangulated module categories (Proposition \ref{prop:kunneth}), thereby extending the version for monoidal triangulated categories (\cite{Big07}). We also prove that a faithful module category over a multiring category gives rise to a faithful module product bifunctor between the corresponding derived categories (Corollary \ref{coro:triangulatedfaithful}). Finally, using the relation between a bounded t-structure with finitely many simple objects up to isomorphism in its heart and general bounded t-structures (Lemma \ref{lem:finitet=t0}), together with the preservation of t-structures by full and dense triangulated functors (Lemma \ref{lem:newtstructure}), we establish Theorem \ref{thm:0}.

Moreover, we investigate the localization theory for module categories and triangulated module categories. We show that the Serre quotient of module categories gives rise to the Verdier quotient of triangulated module categories (Remark \ref{rm:Serre=thick}). Conversely, if the Verdier quotient functor is full and is induced as a triangulated module functor by a monoidal triangulated functor, then the localization of the triangulated module category can induce a localization of the module category (see Theorem \ref{thm:thick=serre}).

This paper is organized as follows. Section \ref{section2} reviews the foundational concepts and properties of the module categories and the triangulated module categories, and investigates their localization theory. Section \ref{section3} develops the theory of compatible t-structures and triangulated module t-structures on triangulated module categories. We establish a general K\"{u}nneth formula on triangulated module categories and prove that a faithful module category over a multiring category gives rise to a faithful module product bifunctor between the corresponding derived categories. In Section \ref{section4}, we begin by examining the localization of triangulated module categories under certain assumptions. We then prove our main result: subcategories of tensor categories, both are abelian categories with finitely many simple objects up to isomorphism, are equivalent as module categories, provided that their derived categories are equivalent as triangulated module categories via an equivalence induced by a monoidal triangulated functor. An application to smash product algebras concludes the section.
\section{Module categories and triangulated module categories}\label{section2}
Throughout this paper $\k$ denotes an \textit{algebraically closed field}.
In this section, we recall some definitions and basic properties related to module categories and triangulated module categories. For further details of triangulated categories, monoidal categories and module categories, we refer the reader to \cite{Hap88} and \cite{EGNO15}.
\subsection{Multiring categories and module categories}
We first introduce the definition of multiring categories.
\begin{definition}\emph{(}\cite[Definition 4.2.3]{EGNO15}\emph{)}
A multiring category over $\k$ is a $\k$-linear abelian monoidal category $\mathcal{C}$ with bilinear and biexact tensor product. If in addition $\operatorname{End}_\mathcal{C}(\mathbbm{1})\cong \k$, then we will call $\mathcal{C}$ a ring category.
\end{definition}
Note that the condition of being ``locally finite" (as in \cite[Definition 4.2.3]{EGNO15}) is omitted from our definition, as it is not essential for the results of this paper. Here by a \textit{locally finite category} (\cite[Definition 1.8.1]{EGNO15})) we mean one whose morphism spaces are finite-dimensional and in which every object has finite length. A locally finite category is said to be \textit{finite} (\cite[Definition 1.8.5]{EGNO15})), provided that it has enough projective objects and finitely many isomorphism classes of simple objects.

Recall that a monoidal category is \textit{rigid} (\cite[Definition 2.10.11]{EGNO15}) if every object has left and right duals in the sense of \cite[Definitions 2.10.1 and 2.10.2]{EGNO15}. We now proceed to introduce the concept of a multitensor category.
\begin{definition}\emph{(}\cite[Definition 4.1.1]{EGNO15}\emph{)}
A locally finite $\k$-linear abelian rigid monoidal category $\mathcal{C}$ is called a multitensor category if the bifunctor $\otimes:\mathcal{C}\times \mathcal{C}\rightarrow \mathcal{C}$ is bilinear on morphisms. If in addition $\operatorname{End}_\mathcal{C}(\mathbbm{1})\cong \k$, $\mathcal{C}$ is called a tensor category.
\end{definition}
It follows from \cite[Proposition 4.2.1]{EGNO15} that every multitensor category is a multiring category, and every tensor category is a ring category.

While the notion of a tensor category categorifies the notion of a ring, a module category provides the categorified analogue of a module over a ring. We now formally introduce its definition.
\begin{definition}\emph{(}\cite[Definition 2.6]{Ost03}\emph{)}
Let $(\mathcal{C},\otimes,\mathbbm{1},a,l,r)$ be a monoidal category. A left module category over $\mathcal{C}$ is a category $\mathcal{M}$ equipped with a module product bifunctor $\otimes:\mathcal{C}\times \mathcal{M}\rightarrow \mathcal{M}$ and functorial associativity and unit isomorphisms $m_{X, Y, M}: (X\otimes Y)\otimes M\rightarrow X\otimes (Y\otimes M), l_M:\mathbbm{1}\otimes M\rightarrow M$ for any $X, Y\in\mathcal{C}, M\in\mathcal{M}$
such that the diagrams
\begin{displaymath}
\xymatrix{
&&\ar[dll]_{a_{X, Y, Z}\otimes \id_M} ((X\otimes Y)\otimes Z)\otimes M\ar[drr]^{m_{X\otimes Y, Z, M}}\\
(X\otimes (Y\otimes Z))\otimes M\ar[d]_{m_{X,Y\otimes Z, M}}&&&& (X\otimes Y)\otimes (Z\otimes M)\ar[d]^{m_{X, Y, Z\otimes M}}\\
X\otimes ((Y\otimes Z)\otimes M)\ar[rrrr]^{id_X\otimes m_{Y,Z,M}} &&&&X\otimes (Y\otimes(Z\otimes M))
}
\end{displaymath}
and
\begin{displaymath}
\xymatrix{
\ar[dr]_{r_X\otimes\id_M}(X\otimes \mathbbm{1})\otimes M\ar[rr]^{m_{X,\mathbbm{1},M}} &&X\otimes (\mathbbm{1}\otimes M)\ar[dl]^{\id_X\otimes l_M}\\
&X\otimes M
}
\end{displaymath}
commute.
\end{definition}

\begin{definition}\emph{(}\cite[Definition 7.2.1]{EGNO15}\emph{)}
Let $\mathcal{M}_1$ and $\mathcal{M}_2$ be two module categories over a monoidal category $\mathcal{C}$ with associativity isomorphisms $m^1$ and $m^2$. A $\mathcal{C}$-module functor from $\mathcal{M}_1$ to $\mathcal{M}_2$ consists of a functor $F:\mathcal{M}_1\rightarrow \mathcal{M}_2$ and a natural isomorphism $s_{X, M}:F(X\otimes M)\rightarrow X\otimes F(M)$, where $X\in \mathcal{C}, M\in\mathcal{M}_1$,
such that the following diagrams
\begin{displaymath}
\xymatrix{
&&\ar[dll]_{F(m^1_{X, Y, M})} F((X\otimes Y)\otimes M)\ar[drr]^{s_{X\otimes Y, M}}\\
F(X\otimes (Y\otimes M))\ar[d]_{s_{X,Y\otimes M}}&&&& (X\otimes Y)\otimes F( M)\ar[d]^{m^2_{X, Y, F(M)}}\\
X\otimes F(Y\otimes M)\ar[rrrr]^{id_X\otimes s_{Y,M}} &&&&X\otimes (Y\otimes F( M))
}
\end{displaymath}
and
\begin{displaymath}
\xymatrix{
\ar[dr]_{F(l_M)}F( \mathbbm{1})\otimes M\ar[rr]^{s_{\mathbbm{1},M}} && \mathbbm{1}\otimes F(M)\ar[dl]^{l_{F(M)}}\\
&F( M)
}
\end{displaymath}
commute. A $\mathcal{C}$-module equivalence $F:\mathcal{M}_1\rightarrow \mathcal{M}_2$ of $\mathcal{C}$-module categories is a module functor $(F, s)$ from $\mathcal{M}_1$ to $\mathcal{M}_2$ such that $F$ is an equivalence of categories.
\end{definition}
While the module product bifunctor for a module category over a general monoidal category has no additional constraints, we will require extra structures when considering module categories over a multiring category for our purposes.
\begin{remark}\rm
In what follows, by a module category over a multiring category $\mathcal{C}$, we mean an abelian category $\mathcal{M}$ endowed with a $\mathcal{C}$-module structure such that the module product bifunctor $\otimes: \mathcal{C} \times \mathcal{M} \to \mathcal{M}$ is bilinear on morphisms and exact in each variable. Moreover, $\mathcal{C}$-module functors are assumed to be exact.
\end{remark}

\begin{example}\rm
Let $\mathcal{C}$ be a multitensor category. According to \cite[Section 4.3]{EGNO15}, the unit object $\mathbbm{1}$ is semisimple, i.e., it decomposes as $\mathbbm{1}=\bigoplus_{i\in I} \mathbbm{1}_i$. Defining $\mathcal{C}_{ij}:=\mathbbm{1}_i\otimes \mathcal{C}\otimes \mathbbm{1}_j$, we obtain a decomposition $\mathcal{C}\cong \bigoplus_{i,j\in I}\mathcal{C}_{ij}$. Here, each $\mathcal{C}_{ii}$ forms a tensor category with unit object $\mathbbm{1}_i$, and for any $i,j\in I$, the category $\mathcal{C}_{ij}$ carries the structure of a left $\mathcal{C}_{ii}$-module category; in fact, it is a $(\mathcal{C}_{ii},\mathcal{C}_{jj})$-bimodule category (see \cite[Example 7.4.5]{EGNO15}).
\end{example}
\begin{example}\rm
Let $H$ be a Hopf algebra over $\k$ (not necessarily finite-dimensional).
Recall that a left $H$-\textit{comodule algebra} $A$ is a $k$-algebra equipped with a left $H$-comodule structure $\rho: A \to H \otimes A$ such that $\rho$ is an algebra homomorphism. Explicitly, this means $\rho(1_A) = 1_H \otimes 1_A$ and $\rho(ab) = \sum a_{(-1)}b_{(-1)} \otimes a_{(0)}b_{(0)}$ for all $a, b \in A$.
Let $A\text{-}\mathrm{mod}$ (resp. $H\text{-}\mathrm{mod}$) denote the category of finite-dimensional left modules over $A$ (resp. over $H$). The tensor product $V\otimes M$ of an $H$-module $V$ and an $A$-module $M$ is naturally an $A$-module, via $\rho.$ Therefore, this tensor product gives rise to a biexact bifunctor
$$
H\text{-}\mathrm{mod} \times A\text{-}\mathrm{mod} \rightarrow A\text{-}\mathrm{mod},
$$
which endows $A\text{-}\mathrm{mod}$ with the structure of a module category over $ H\text{-}\mathrm{mod}$.
\end{example}

Recall that for an abelian category $\mathcal{A}$, a full additive subcategory $\mathcal{B}$ of $\mathcal{A}$ is called a \textit{Serre subcategory} (\cite{Pop73}) if $\mathcal{B}$ is closed under taking subobjects, quotients and extensions. A Serre subcategory $\mathcal{I}$ of a
multiring category $\mathcal{C}$ is called a \textit{left (resp. right) Serre tensor ideal} (\cite[Subsection 3.2]{XZ25}) if $X\otimes I\in\mathcal{I}$ (resp. $I\otimes X\in\mathcal{I}$) for any $X\in \mathcal{C}, I\in\mathcal{I}.$ If $\mathcal{I}$ is both a left and right Serre tensor ideal, then it is called a \textit{two-sided Serre tensor ideal}.
\begin{lemma}\label{lem:leftserre}
Let $\mathcal{C}$ be a multiring category and $\mathcal{M}$ be a left $\mathcal{C}$-module category. For any nonzero object $M\in\mathcal{M}$, let $\operatorname{Ann}_{\mathcal{C}}(M):=\{X\in\mathcal{C}\mid X\otimes M\cong0\}$ be the full subcategory of $\mathcal{C}$. Then $\operatorname{Ann}_{\mathcal{C}}(M)$ is a left Serre tensor ideal of $\mathcal{C}$.
\end{lemma}
\begin{proof}
We first take any short exact sequence $$0\rightarrow X\rightarrow Y\rightarrow Z\rightarrow 0$$ in $\mathcal{C}$ and apply the exact functor $-\otimes M$ to it. It can be shown that $Y\otimes M\cong 0$ if and only if $X\otimes M\cong Z\otimes M\cong0$. Consequently, $\operatorname{Ann}_{\mathcal{M}}(X)$ is a Serre subcategory. Moreover, for any $W\in\mathcal{C}, X\in \operatorname{Ann}_{\mathcal{C}}(M)$, we have $(W\otimes X) \otimes M\cong W\otimes (X\otimes M)\cong0,$ which implies that $\operatorname{Ann}_{\mathcal{C}}(M)$ is a left Serre tensor ideal.
 \end{proof}
Motivated by \cite[Definition 7.12.9]{EGNO15}, we adapt the definition to our setting as follows.
\begin{definition}
A module category $\mathcal{M}$ over a monoidal category $\mathcal{C}$ is said to be faithful if any nonzero object in $\mathcal{C}$ acts by a nonzero functor in $\mathcal{M}.$
\end{definition}

\begin{corollary}\label{coro:tensorfaithful}
Let $\mathcal{C}$ be a tensor category and $\mathcal{M}$ be a left $\mathcal{C}$-module category. Then $\mathcal{M}$ is faithful.
\end{corollary}
\begin{proof}
For any nonzero object $M\in\mathcal{M}$, it follows from Lemma \ref{lem:leftserre} that $\operatorname{Ann}_{\mathcal{C}}(M):=\{X\in\mathcal{C}\mid X\otimes M=0\}$ is a left Serre tensor ideal of $\mathcal{C}$. Following the argument in the proof of \cite[Proposition 3.11]{XZ25}, the proof is complete.
\end{proof}
We now turn to Serre subcategories of module categories over a multiring category.
\begin{definition}
Let $\mathcal{C}$ be a multiring category and $\mathcal{M}$ be a left $\mathcal{C}$-module category. A Serre subcategory $\mathcal{N}$ of $\mathcal{M}$ is called a Serre submodule category if $X\otimes N \in \mathcal{N}$ for any $X\in\mathcal{C}$ and $N\in\mathcal{N}$.
\end{definition}
\begin{example}\rm
Let $\mathcal{C}$ be a braided multiring category (\cite[Definition 8.1.1]{EGNO15}) and $\mathcal{M}$ be a left $\mathcal{C}$-module category. For any nonzero object $X\in \mathcal{C}$, let $\operatorname{Ann}_{\mathcal{M}}(X):=\{M\in\mathcal{M}\mid X\otimes M=0\}$ be the full subcategory of $\mathcal{M}$. A similar argument as in Lemma \ref{lem:leftserre} shows that $\operatorname{Ann}_{\mathcal{M}}(X)$ is a Serre submodule of $\mathcal{M}$.
\end{example}
Let $\mathcal{A}$ be an abelian category and $\mathcal{B}$ be its Serre subcategory. Recall that we can form the the \textit{Serre quotient} (\cite{Pop73}) of $\mathcal{A}$ by $\mathcal{B}$ by inverting each morphism
$f$ in $\mathcal{A}$ such that its kernel and cokernel belong to $\mathcal{B}$. The Serre quotient $\mathcal{A} /\mathcal{B}$ is also an abelian
category and the quotient functor $Q : \mathcal{A} \rightarrow\mathcal{A} /\mathcal{B}$ is an exact functor.

In \cite[Proposition 4.9]{ZL25}, it is shown that the Serre quotient $\mathcal{C} / \mathcal{I}$ of a multiring category $\mathcal{C}$ by a two-sided Serre tensor ideal $\mathcal{I}$ itself carries the structure of a multiring category. By an analogous argument, one can prove that the Serre quotient $\mathcal{M} / \mathcal{N}$ of a $\mathcal{C}$-module category $\mathcal{M}$ by a Serre submodule $\mathcal{N}$ also naturally inherits the structure of a $\mathcal{C}$-module category. Moreover, if for any $I\in \mathcal{I}$ and $M\in\mathcal{M}$, $I\otimes M\in\mathcal{N}$, then $\mathcal{M} / \mathcal{N}$ is a $\mathcal{C} / \mathcal{I}$-submodule.

\subsection{Triangulated module categories}
According to \cite{NVY22}, a \textit{monoidal triangulated category} $(\mathcal{T},\otimes,[1],\mathbbm{1})$ consists of a triangulated category $(\mathcal{T},[1])$, a monoidal structure $\otimes$, and a unit object $\mathbbm{1}$, where the bifunctor $\otimes$ is exact in each variable.
A \textit{monoidal triangulated functor} (\cite{NVY22}) is a functor that is both triangulated and monoidal.

\begin{example}\emph{(}\cite[Example 2.7]{XZ25}\emph{)}\rm
Given a multiring category $(\mathcal{C}, \otimes, \mathbbm{1})$, the bounded homotopy category $\mathbf{K}^{b}(\mathcal{C})$ admits a monoidal structure. The tensor product in $\mathbf{K}^{b}(\mathcal{C})$ is defined as follows:
for any complexes $X^\bullet=(X^n, d^n_X)_{n\in\mathbb{Z}}, Y^\bullet=(Y^m, d^m_Y)_{m\in\mathbb{Z}}\in \mathbf{K}^{b}(\mathcal{C})$,  $$(X^\bullet\widetilde{\otimes} Y^\bullet)^n:=\bigoplus_{i+j=n}X^i\otimes Y^{j}$$
with differential
$$d^n_{X^{\bullet}\widetilde{\otimes} Y^{\bullet}}:=\sum\limits_{i+j=n}(d^i_{X}\otimes \id_{Y^j} +(-1)^i\id_{X^i}\otimes d_Y^j).$$
Let $[1]$ denote the shift functor of $\mathbf{K}^{b}(\mathcal{C})$.
Then $(\mathbf{K}^{b}(\mathcal{C}), \widetilde{\otimes}, [1], \mathbbm{1}^{\bullet})$ forms a monoidal triangulated category, where $\mathbbm{1}^{\bullet}$ is the stalk complex with $\mathbbm{1}$ concentrated in degree $0$. According to \cite[Lemma 2.7.3]{Wei94}, the bounded derived category $(\mathbf{D}^{b}(\mathcal{C}), \widetilde{\otimes}, [1], \mathbbm{1}^{\bullet})$ with the inherited monoidal structure is also a monoidal triangulated category.
\end{example}
Inspired by \cite[Definitions 3.2 and 3.4]{Ste13}, we introduce the following definitions.
\begin{definition}
Let $(\mathcal{T},\otimes,\Sigma,\mathbbm{1})$ be a monoidal triangulated category.
\begin{itemize}
\item[(1)]A left triangulated module category over $\mathcal{T}$ is a triangulated category $\mathcal{K}$ which is equipped with a structure of a $\mathcal{T}$-module category, such that the module product bifunctor $\otimes :\mathcal{T}\times\mathcal{K}\rightarrow \mathcal{K}$ is exact in both factors.
\item[(2)]A triangulated $\mathcal{T}$-submodule $\mathcal{K}^\prime$ of a triangulated $\mathcal{T}$-module category $\mathcal{K}$ is a triangulated subcategory which is closed under the action of $\mathcal{T}.$
\end{itemize}
\end{definition}
Observe that in contrast to \cite[Definition 3.2]{Ste13}, our definition of a triangulated module category omits those further assumptions.

\begin{definition}
Let $\mathcal{K}_1, \mathcal{K}_2$ be two triangulated module categories over a monoidal triangulated category $\mathcal{T}$. A triangulated $\mathcal{T}$-module functor from $\mathcal{K}_1$ to $\mathcal{K}_2$ is a functor that is simultaneously a triangulated functor and a $\mathcal{T}$-module functor. A triangulated $\mathcal{T}$-module functor is said to be an equivalence of triangulated module categories if it is an equivalence of ordinary categories.
\end{definition}

\begin{example}\rm
Let $F:\mathcal{T}_1\rightarrow \mathcal{T}_2$ be a monoidal triangulated functor between monoidal triangulated categories $\mathcal{T}_1$ and $\mathcal{T}_2$. Then $\mathcal{T}_2$ has a structure of a triangulated module category over $\mathcal{T}_1$ with $X\otimes Y:=F(X)\otimes Y$ for any $X\in \mathcal{T}_1, Y\in\mathcal{T}_2.$ In particular, any monoidal triangulated category is a triangulated module category over itself.
\end{example}
Indeed, given a module category over a multiring category, one can construct a triangulated module category over a monoidal triangulated category.
\begin{example}\label{ex:exact}\rm
Let $\mathcal{C}$ be a non-semisimple finite multitensor category. Suppose $\mathcal{M}$ is an exact module category over $\mathcal{C}$, that is, for any projective object $P\in \mathcal{C}$ and any object $M\in \mathcal{M}$ the object $P\otimes M$ is projective in $\mathcal{M}$ (\cite[Definition 7.5.1]{EGNO15}). Since $\mathcal{C}$ itself is a Frobenius categroy by \cite[Remark 6.1.4]{EGNO15}, and the stable category of a Frobenius category is a triangulated category (\cite[Theorem 2.6]{Hap88}), then the stable category of $\mathcal{C}$ carries the structure of a monoidal triangulated category.
Note that by \cite[Lemma 7.6.1 and Corollary 7.6.4]{EGNO15}, $\mathcal{M}$ is also Frobenius; hence we may form its stable category $\underline{\mathcal{M}}$.
For any object $X\in \mathcal{C}$ and any projective object $Q\in\mathcal{M}$, the object $X\otimes Q$ is always projective in $\mathcal{M}$ (see \cite[Exercise 7.5.2]{EGNO15}). It is striaghtforward to show that the stable category $\underline{\mathcal{M}}$ is a triangulated module category over the stable category $\underline{\mathcal{C}}$.
\end{example}

\begin{example}\rm\label{ex:derived}
Let $\mathcal{C}$ be a multiring category and $\mathcal{M}$ a left $\mathcal{C}$-module category. The bounded complex category $\mathbf{C}^{b}(\mathcal{C})$ is clear a multiring category and it has a module action on the bounded complex category $\mathbf{C}^{b}(\mathcal{M})$: for any complexes $X^\bullet=(X^n, d^n_X)_{n\in\mathbb{Z}}\in \mathbf{C}^{b}(\mathcal{C})$ and $M^\bullet=(M^m, d^m_M)_{m\in\mathbb{Z}}\in\mathbf{C}^b(\mathcal{M})$,  $$(X^\bullet\widetilde{\otimes} M^\bullet)^n:=\bigoplus_{i+j=n}X^i\otimes M^{j}$$
with differential
$$d^n_{X^{\bullet}\widetilde{\otimes} M^{\bullet}}:=\sum\limits_{i+j=n}(d^i_{M}\otimes \id_{M^j} +(-1)^i\id_{X^i}\otimes d_M^j).$$
Since a direct computation shows that the tensor product $\tilde{\otimes}$ preserves null-homotopies, $\mathbf{K}^b(\mathcal{M})$ becomes a triangulated module category over $\mathbf{K}^b(\mathcal{C})$. Moreover, $\tilde{\otimes}$ also preserves quasi-isomorphisms by an argument similar to the proof of \cite[Lemma 2.7.3]{Wei94}. Consequently, the structure descends to make $\mathbf{D}^b(\mathcal{M})$ a triangulated module category over $\mathbf{D}^b(\mathcal{C})$.
\end{example}

\begin{remark}\rm
In \cite{Kho16}, Khovanov introduced the notion of Hopfological algebra. Subsequently, Qi developed its homological theory \cite{Qi14}, drawing an analogy with the classical theory of ordinary differential graded algebras. These frameworks can be applied to areas such as the categorification of link invariants as well as their representation (see, for example, \cite{EQ16, QS16}). In what follows, we briefly review their results. Let $H$ be a finite-dimensional Hopf algebra over $\k$ and $A$ be a left $H$-comodule algebra. Denote $\mathcal{C}(A, H)$ by the quotient of $A\text{-}\mathrm{mod}$ by the ideal of morphisms that factor through an $A\text{-}\mathrm{mod}$ of the form $H\otimes N$ for some $N\in A\text{-}\mathrm{mod}$. By \cite[Theorem 1]{Kho16} and \cite[Proposition 2.6]{Qi23}, the category $\mathcal{C}(A, H)$ is triangulated and also carries a module structure over the stable category $H\text{-}\underline{\mathrm{mod}}$, with a module product bifunctor that is exact in the second variable. They also employed the notion of a triangulated module category and regarded $\mathcal{C}(A, H)$ as a triangulated module category over $H\text{-}\underline{\mathrm{mod}}$. However, compared to our definition, theirs lacks the condition of exactness in the first component. If, in addition, $A\text{-}\mathrm{mod}$ is an exact module category over $H\text{-}\mathrm{mod}$ (see Example \ref{ex:exact}), then it follows from \cite[Theorem 2.9]{Qi23} that the bifunctor is also exact in the first variable. It is worth noting that the classification of indecomposable exact module categories over $H\text{-}\mathrm{mod}$ was established in \cite{AM07}.
\end{remark}

\begin{lemma}\label{lem:DF}
Let $\mathcal{M}_1, \mathcal{M}_2$ be two module categories over a multiring category $\mathcal{C}$. If $F:\mathcal{M}_1\rightarrow \mathcal{M}_2$ is an exact $\mathcal{C}$-module functor, then the derived functor $\mathbf{D}^b(F)$ of $F$ is a $\mathbf{D}^b(\mathcal{C})$-module functor. In particular, if $F$ is an equivalence, then $\mathbf{D}^b(F)$ is an equivalence.
\end{lemma}
\begin{proof}
Since $F$ is exact, it induces a triangulated functor $\mathbf{D}^b(F): \mathbf{D}^b(\mathcal{M}_1) \rightarrow \mathbf{D}^b(\mathcal{M}_2)$. Explicitly, it sends a complex $(M^m, d^m_M)_{m \in \mathbb{Z}}$ to the complex $(F(M^m), F(d^m_M))_{m \in \mathbb{Z}}$. For any $X^\bullet \in\mathbf{D}^b(\mathcal{C})$ and $M^\bullet\in \mathbf{D}^b(\mathcal{M}_1),$ we have
\begin{equation*}
(\mathbf{D}^b(F)(X^\bullet\tilde{\otimes} M^\bullet))^n=   F(\coprod_{p+q=n} X^p\otimes M^q)
\cong
\coprod_{p+q=n} X^p\otimes F(M^q)=
    (X^\bullet\tilde{\otimes} \mathbf{D}^b(F)(M^\bullet))^n
\end{equation*}
for any $n\in\mathbb{Z}.$ This yields an isomorphism $\mathbf{D}^b(X^\bullet\tilde{\otimes} M^\bullet)\cong X^\bullet \tilde{\otimes} \mathbf{D}^b(F)(M^\bullet)$, which holds naturally in both variables. Thus $\mathbf{D}^b(F)$ is a triangulated $\mathbf{D}^b(\mathcal{C})$-module functor. If $F$ is an equivalence, then any quasi-inverse of $F$ induces a quasi-inverse for $\mathbf{D}^b(F)$; hence $\mathbf{D}^b(F)$ is also an equivalence.
\end{proof}
Recall that a \textit{thick subcategory}
(\cite[Section 1]{Ric89}) of a triangulated category is a triangulated subcategory closed under taking direct summands.
We now present the corresponding definitions in the context of monoidal triangulated categories and triangulated module categories.
\begin{definition}\emph{(}\cite[Subsection 1.2]{NVY22}, \cite[Definition 3.4]{Ste13}\emph{)}
Let $\mathcal{T}$ be a monoidal triangulated category and $\mathcal{K}$ be a triangulated module category over $\mathcal{T}$.
\begin{itemize}
\item[(1)]A thick subcategory $\mathcal{J}$ of $\mathcal{T}$ is said to be a two-sided thick ideal if it satisfies the ideal condition: for each $J\in \mathcal{J}$ and $X\in\mathcal{T}$, $J\otimes X, X\otimes J\in \mathcal{J}$.
\item[(2)]A thick subcategory $\mathcal{L}$ of $\mathcal{K}$ is called a thick $\mathcal{T}$-submodule category provided that $X\otimes L\in \mathcal{L}$ for any $X\in\mathcal{T}$ and $L\in\mathcal{L}$.
\end{itemize}
\end{definition}

\begin{remark}\rm
It was proved by Verdier \cite{Ver77} that there exists a bijection between the thick subcategories and the saturated compatible multiplicative systems in any triangulated category. In fact, this result also admits a generalization to the setting of monoidal triangulated categories and triangulated module categories. More specifically, there is a one-to-one correspondence between two-sided thick ideals (resp. thick submodule categories) and saturated compatible multiplicative systems that are closed under taking tensor product (resp. closed under tensoring with the identity morphisms in the monoidal triangulated category) in a monoidal triangulated category (resp. triangulated module category).
\end{remark}

For a triangulated category $\mathcal{K}$ and its thick subcategory $\mathcal{E}$, we can define the \textit{Verdier quotient} $\mathcal{K}/\mathcal{E}$, which
is the localization of $\mathcal{K}$ by inverting all morphisms $f$ in $\mathcal{K}$ whose cones lie in $\mathcal{E}$. The Verdier quotient $\mathcal{K}/\mathcal{E}$ is still a triangulated category, and the quotient functor $\mathcal{K}\rightarrow \mathcal{K}/\mathcal{E}$ is a triangulated functor. See \cite[Chapter 2]{Nee01} for details.

Let $\mathcal{T}$ be a monoidal triangulated category and $\mathcal{K}$ be a triangulated module category over $\mathcal{T}$.
It is straightforward to show that the Verdier quotient $\mathcal{T}/\mathcal{J}$ (resp. $\mathcal{K}/\mathcal{L}$) of $\mathcal{T}$ (resp. $\mathcal{K}$) with respect to its two-sided thick ideal $\mathcal{J}$ (resp. thick $\mathcal{T}$-submodule category  $\mathcal{L}$) is still a monoidal triangulated category (resp. triangulated module category over $\mathcal{T}$).

\begin{lemma}\label{lem:k1/ker=k2}
Let $\mathcal{C}$ be a multiring category and $\mathcal{M}_1, \mathcal{M}_2$ be two module categories over $\mathcal{C}$.
Suppose that $F: \mathbf{D}^{b}(\mathcal{M}_1)\rightarrow \mathbf{D}^{b}(\mathcal{M}_2)$ is a triangulated $\mathbf{D}^{b}(\mathcal{C})$-module functor.
\begin{itemize}
\item [(1)]Then $\ker(F)$ is a thick $\mathbf{D}^{b}(\mathcal{C})$-submodule category.
 \item[(2)]If $F$ is full and dense, then we have $\mathbf{D}^{b}(\mathcal{M}_1)/\ker(F)$ is equivalent to $\mathbf{D}^{b}(\mathcal{M}_2)$ as triangulated $\mathbf{D}^{b}(\mathcal{C})$-modules categories.
\end{itemize}
\end{lemma}
\begin{proof}
\begin{itemize}
\item[(1)]Since $F$ is a triangulated functor, it follows that $\ker(F)$ is a thick subcategory. Moreover, for any $M\in\ker(F)$ and $X\in \mathbf{D}^{b}(\mathcal{C})$, we have $$F(X\otimes M)\cong X\otimes F(M)\cong0,$$ which means that $\ker(F)$ is a thick $\mathbf{D}^{b}(\mathcal{C})$-submodule category.
\item[(2)]Given the Verdier quotient functor $Q:\mathbf{D}^{b}(\mathcal{M}_1) \rightarrow \mathbf{D}^{b}(\mathcal{M}_1)/\ker(F)$, its universal property (\cite[Theorem 2.1.8]{Nee01}) guarantees the existence of a triangulated $\mathbf{D}^{b}(\mathcal{C})$-module functor $\overline{F}:\mathbf{D}^{b}(\mathcal{M}_1)/\ker(F) \rightarrow \mathbf{D}^{b}(\mathcal{M}_2)$ such that the following diagram
\begin{displaymath}
\xymatrix{
\mathbf{D}^{b}(\mathcal{M}_1)\ar[d]^{Q}\ar[rr]^F&&\mathbf{D}^{b}(\mathcal{M}_2)\\
\mathbf{D}^{b}(\mathcal{M}_1)/\operatorname{ker}(F)\ar[urr]_{\overline{F}}
}
\end{displaymath}
commutes. Since $\overline{F}Q=F$, it follows that $\overline{F}$ is full and dense.
Using \cite[Lemma 3.1]{Miy91}, one can show that $\overline{F}$ is a triangulated $\mathbf{D}^{b}(\mathcal{C})$-modules equivalence.
\end{itemize}
\end{proof}

\begin{remark}\rm\label{rm:Serre=thick}
Let $\mathcal{C}$ be a multiring category and $\mathcal{M}$ be a left $\mathcal{C}$-module category. Suppose $\mathcal{N}$ is a Serre submodule category of $\mathcal{M}$. Using \cite[Theorem 3.2]{Miy91}, we can show that $$ \mathbf{D}^{b}(\mathcal{M})/\mathbf{D}^{b}_{\mathcal{N}}(\mathcal{M})\approx \mathbf{D}^{b}(\mathcal{M}/\mathcal{N})$$ as triangulated module categories over $\mathbf{D}^b(\mathcal{C})$, where $\mathbf{D}^{b}_{\mathcal{N}}(\mathcal{M})$ is a full subcategory of $\mathbf{D}^{b}(\mathcal{M})$ generated by complexes of which all homologies are in $\mathcal{N}.$ A corresponding version of the statement holds for monoidal triangulated categories.
\end{remark}
It is now straightforward to prove the following lemma.
\begin{lemma}\label{lem:IL=L}
Let $\mathcal{T}$ be a monoidal triangulated category and $\mathcal{K}$ be a triangulated module category over $\mathcal{T}$. Let $\mathcal{J}$ be a two-sided thick ideal of $\mathcal{T}$ and $\mathcal{L}$ a thick $\mathcal{T}$-submodule category of $\mathcal{K}$. If $J\otimes M\in\mathcal{L}$ for all $J\in \mathcal{J}$ and $M\in\mathcal{K}$, then $\mathcal{K}/\mathcal{L}$ is a triangulated module category over $\mathcal{T}/\mathcal{J}$.
\end{lemma}

\begin{example}\rm
Let $\mathcal{C}$ be a finite multitensor category and $\mathcal{M}$ be a exact left $\mathcal{C}$-module category. By Example \ref{ex:exact}, one can show that $\mathbf{K}^b(\mathcal{P}_\mathcal{M})$ is a thick $\mathbf{D}^b(\mathcal{C})$-submodule category of $\mathbf{D}^b(\mathcal{M})$ and $\mathbf{K}^b(\mathcal{P}_\mathcal{C})$ is a two-sided thick ideal of $\mathbf{D}^b(\mathcal{C})$, where $\mathcal{P}_{\mathcal{M}}$ and $\mathcal{P}_{\mathcal{N}}$ denote the full subcategories of $\mathcal{M}$ and $\mathcal{N}$, respectively, consisting of projective objects. Moreover, it follows from Lemma \ref{lem:IL=L} that $\mathbf{D}^b(\mathcal{M})/\mathbf{K}^b(\mathcal{P_\mathcal{M}})$ is a triangulated module category over $\mathbf{D}^b(\mathcal{C})/\mathbf{K}^b(\mathcal{P_\mathcal{C}})$.
\end{example}

\section{t-sturctures on triangulated module categories}\label{section3}
In this section, we present the theory of t-structures on triangulated module categories, with a focus on compatible t-structures and triangulated module t-structures, and investigate their properties.
\subsection{Preliminaries on t-structures}
Let $(\mathcal{K}, [1])$ be a triangulated category. Recall from \cite{BBD82} that
a \textit{t-structure} $\mathbbm{t}=(\mathcal{K}^{\leq 0}, \mathcal{K}^{\geq 0})$ on $\mathcal{K}$ is a pair of full subcategories satisfying the following conditions:
      \begin{itemize}
        \item [(i)]$\mathcal{K}^{\leq 0} \subseteq \mathcal{K}^{\leq 1}$ and $\mathcal{K}^{\geq 1}\subseteq \mathcal{K}^{\geq 0}$ where we use notation $\mathcal{K}^{\leq n}=\mathcal{K}^{\leq 0}[-n]$ and $\mathcal{K}^{\geq n}=\mathcal{K}^{\geq 0}[-n]$;
        \item [(ii)]If $U\in \mathcal{K}^{\leq 0}$ and $V\in \mathcal{K}^{\geq 1}$, then $\operatorname{Hom}_{\mathcal{K}}(U, V)=0;$
        \item [(iii)]For any object $X\in\mathcal{K}$, there is a distinguished triangle
        $$U \rightarrow X\rightarrow V \rightarrow U[1]$$
        with $U\in \mathcal{K}^{\leq 0}$ and $V\in\mathcal{K}^{\geq 1}$.
      \end{itemize}

Let $\mathbbm{t}=(\mathcal{K}^{\leq 0}, \mathcal{K}^{\geq 0})$ be a t-structure on a triangulated category $\mathcal{K}$. We observe that $(\mathcal{K}^{\leq n}, \mathcal{K}^{\geq n})$ is still a t-stucture for any $n\in\mathbb{Z}$, which will be denoted by $\mathbbm{t}[-n]$.

The \textit{heart} of the t-structure $\mathbbm{t}$ is the full subcategory
  $\mathcal{K}^{\leq 0} \cap \mathcal{K}^{\geq 0}$, which is denoted by $\mathcal{H}_\mathbbm{t}$. As is well-known, $\mathcal{H}_\mathbbm{t}$ is an abelian category, see \cite[Theorem 1.3.6]{BBD82}.
A t-structure $(\mathcal{K}^{\leq 0}, \mathcal{K}^{\geq 0})$ is said to be \textit{bounded} if $$\bigcup\limits_{n\in\mathbb{Z}}\mathcal{K}^{\leq n}=\bigcup\limits_{n\in\mathbb{Z}}\mathcal{K}^{\geq n}=\mathcal{K}.$$

Denote by $\tau_{\leq n}:\mathcal{K}\rightarrow \mathcal{K}^{\leq n}$ the right adjoint of the inclusion $\mathcal{K}^{\leq n}\hookrightarrow \mathcal{K}$, and by $\tau_{\geq n}:\mathcal{K}\rightarrow \mathcal{K}^{\geq n}$ the left adjoint of the inclusion $\mathcal{K}^{\geq n}\hookrightarrow \mathcal{K}$.
They are called the \textit{truncation functors} associated to the t-structure $\mathbbm{t}$. For each object $X\in\mathcal{K}$, there is a canonical distinguished triangle
$$
\tau_{\leq n}X\rightarrow X\rightarrow \tau_{\geq n+1} X\rightarrow (\tau_{\leq n}X)[1].
$$
The composition $H^0_t=\tau_{\leq 0}\tau_{\geq 0}:\mathcal{K} \rightarrow \mathcal{H}_\mathbbm{t}$ is called the \textit{cohomological functor} associated to $\mathbbm{t}$. More generally, we set $H^n_\mathbbm{t}(X)=H^0(X[n])$, which is canonically isomorphic to $(\tau_{\leq n}\tau_{\geq n} X)[n]$. For further details,  we refer to \cite[Section 1.3]{BBD82}.

\subsection{K\"{u}nneth formula for triangulated module categories}
In this subsection, let $\mathcal{T}$ be a monoidal triangulated category equipped with a  t-structure $\mathbbm{t}$, and let $\mathcal{K}$ be a triangulated module category over $\mathcal{T}$, endowed with a t-structure $\mathbbm{t}^\prime$.

Recall that the tensor product on $\mathcal{T}$ is said to be \textit{compatible} with $\mathbbm{t}$ if $\mathcal{H}_\mathbbm{t}\otimes \mathcal{H}_\mathbbm{t}\subseteq \mathcal{H}_\mathbbm{t}$ (\cite[Definition 3.2]{Big07}). Next, we extend this definition to the setting of triangulated module categories.
\begin{definition}
Let $\mathcal{T}$ be a monoidal triangulated category equipped with a compatible t-structure $\mathbbm{t}$, and let $\mathcal{K}$ be a triangulated module category over $\mathcal{T}$, endowed with a t-structure $\mathbbm{t}^\prime$. The module product bifunctor is said to be compatible with the t-structures $\mathbbm{t}$ and $\mathbbm{t}^\prime$ if $\mathcal{H}_\mathbbm{t}\otimes \mathcal{H}_{\mathbbm{t}^\prime}\subseteq \mathcal{H}_{\mathbbm{t}^\prime}.$
\end{definition}

In the following part, we denote the truncation functors on $\mathcal{T}$ by $\tau_{\leq n}, \tau_{\geq n}$, and those on $\mathcal{K}$ by $\tau^\prime_{\leq n}, \tau^\prime_{\geq n}$.
Following exactly the same method as in \cite[Lemma 3.4]{Big07}, we can prove the following lemma in a completely analogous manner.
\begin{lemma}\label{lem:tauprime=0}
Let $\mathcal{T}$ be a monoidal triangulated category equipped with a compatible bounded t-structure $\mathbbm{t}$, and let $\mathcal{K}$ be a triangulated module category over $\mathcal{T}$, endowed with a bounded t-structure $\mathbbm{t}^\prime$. Then the following assertions are equivalent.
\begin{itemize}
\item[(1)]The module product bifunctor is compatible with the t-structures $\mathbbm{t}$ and $\mathbbm{t}^\prime$, that is, $\mathcal{H}_\mathbbm{t}\otimes \mathcal{H}_{\mathbbm{t}^\prime}\subseteq \mathcal{H}_{\mathbbm{t}^\prime}.$
\item[(2)]For any $X\in\mathcal{T}, M\in\mathcal{K}$ and integers $n,m$, we have $\tau^\prime_{\geq n+m+1}(\tau_{\leq n}X\otimes \tau^\prime_{\leq m}M)=0$ and $\tau^\prime_{\leq n+m-1}(\tau_{\geq n }X\otimes \tau^\prime_{\geq m}M)=0$.
\end{itemize}
\end{lemma}
We now establish a general K\"{u}nneth formula relating the cohomology associated with a t-structure on a monoidal triangulated category and that on a triangulated module category. This result generalizes the corresponding formula for monoidal triangulated categories (\cite[Theorem 4.1]{Big07}).
\begin{proposition}\label{prop:kunneth}
Let $\mathcal{T}$ be a monoidal triangulated category equipped with a compatible bounded t-structure $\mathbbm{t}$, and let $\mathcal{K}$ be a triangulated module category over $\mathcal{T}$ endowed with a bounded t-structure $\mathbbm{t}^\prime$. Suppose that the module product bifunctor is compatible with $\mathbbm{t}$ and $\mathbbm{t}^\prime$. Then for objects $X\in\mathcal{T}$, $M\in\mathcal{K}$ and an integer $n$, there is a natural isomorphism
$$
H^n_{\mathbbm{t}^\prime}(X\otimes M)\cong \coprod\limits_{p+q=n} H^p_{\mathbbm{t}^\prime}(X)\otimes H^q_{\mathbbm{t}^\prime}(M).
$$
\end{proposition}
\begin{proof}
This can be established via an argument analogous to that of \cite[Theorem 4.1]{Big07}. Nevertheless, we provide the proof below for the sake of completeness and reader convenience. For any pair $p, q$ of integers with $p+q=n$, we consider the distinguished triangle
$$\tau^\prime_{\leq n}(\tau_{\leq p} X\otimes \tau^\prime_{\leq q} M)\rightarrow \tau_{\leq p} X\otimes \tau^\prime_{\leq q} M \rightarrow \tau^\prime_{\geq n+1}(\tau_{\leq p} X\otimes \tau^\prime_{\leq q} M)\rightarrow \tau^\prime_{\leq n}(\tau_{\leq p} X\otimes \tau^\prime_{\leq q} M)[1].$$
By Lemma \ref{lem:tauprime=0}, we have
$\tau^\prime_{\leq n}(\tau_{\leq p} X\otimes \tau^\prime_{\leq q} M)\cong \tau_{\leq p} X\otimes \tau^\prime_{\leq q} M.$
The definition of $\tau_{\leq n}^\prime$ gives a natural morphism $\rho:\tau_{\leq p} X\otimes \tau^\prime_{\leq q} M \rightarrow \tau^\prime_{\leq n}(X\otimes M).$
Next we consider the distinguished triangle
$$
\tau_{\leq p-1}X\rightarrow \tau_{\leq p} X\rightarrow H^p_\mathbbm{t}(X)[-p]\rightarrow (\tau_{\leq p-1}X)[1]
$$
in $\mathcal{T}$
and apply the triangulated functor $-\otimes \tau^\prime_{\leq q} M$ to it. From Lemma \ref{lem:tauprime=0}, the long exact cohomology sequence associated with the resulting triangle shows that $H^n_{\mathbbm{t}^\prime}(\tau_{\leq p }X\otimes \tau^\prime_{\leq q}M)$ is naturally isomorphic to $ H^n_{\mathbbm{t}^\prime}(H^p_\mathbbm{t}(X)[-p]\otimes \tau_{\leq q}^\prime M).$
Similarly, applying the triangulated functor $H^p_\mathbbm{t}(X)[-p]\otimes -$ to the distinguished triangle
$$
\tau^\prime_{\leq q-1}M\rightarrow \tau^\prime_{\leq q} M\rightarrow H^q_{\mathbbm{t}^\prime}(M)[-q]\rightarrow (\tau_{\leq q-1}M)[1]
$$
yields that the natural morphim $$k_n:=H^n_{\mathbbm{t}^\prime} (\tau_{\leq p }X\otimes \tau^\prime_{\leq q}M \rightarrow H^p_\mathbbm{t}(X)[-p]\otimes H^q_{\mathbbm{t}^\prime}(M)[-q])$$ is an isomorphism. We now define the morphism ${}_{p}\cap_{q}$ as the composition $H^n(\rho)\circ k_n^{-1}.$ Then, by taking the coproduct over $p+q=n$ of these components, we obtain the induced morphism
$$\cap:\coprod\limits_{p+q=n} H^p_\mathbbm{t}(X)\otimes H^q_{\mathbbm{t}^\prime}(M) \rightarrow H^n_{\mathbbm{t}^\prime}(X\otimes M),$$ which is natural in both arguments. Note that $\cap$ becomes a natural isomorphism if either $X \cong A[-p]$ for some $A \in \mathcal{H}_\mathbbm{t}$, or $M \cong B[-q]$ for some $B \in \mathcal{H}_{\mathbbm{t}^\prime}$. Next we show that $\cap$ is an isomorphism for all $X\in\mathcal{T}$ and $M\in\mathcal{K}$. For a fix object $M\in\mathcal{K}$, define the class $\mathcal{S}$ to be the full subcategory of $\mathcal{T}$ consisting of objects $X$ for which the morphism $\cap$ is an isomorphism for all integers $n$. Clearly, $\mathcal{S}$ is invariant under shifts and finite coproducts. Applying the triangulated functor $-\otimes H^q_{\mathbbm{t}^\prime}(M)$ to the distinguished triangle
$$
U\rightarrow V\rightarrow X\rightarrow U[1]$$ in $\mathcal{T}$,
where $U, V\in\mathcal{S},$ we obtain a new distinguished triangle
$$
U\otimes H^q_{\mathbbm{t}^\prime}(M)\rightarrow V\otimes H^q_{\mathbbm{t}^\prime}(M)\rightarrow X\otimes H^q_{\mathbbm{t}^\prime}(M)\rightarrow (U\otimes H^q_{\mathbbm{t}^\prime}(M))[1].
$$
Since $H^q_{\mathbbm{t}^\prime}(Y)\in \mathcal{H}_{\mathbbm{t}^\prime}$, we see that the long exact sequence of cohomology objects corresponding to the latter distinguished triangle reads as
\begin{eqnarray*}
\cdots &\rightarrow& H^p_{\mathbbm{t}}(U)\otimes H^q_{\mathbbm{t}^\prime}(M)\rightarrow H^p_\mathbbm{t}(V) \otimes H^q_{\mathbbm{t}^\prime}(M)\rightarrow H^p_{\mathbbm{t}}(X)\otimes H^q_{\mathbbm{t}^\prime}(M) \\
&&\rightarrow H^{p+1}_\mathbbm{t}(U)\otimes H^q_{\mathbbm{t}^\prime}(M) \rightarrow H^{p+1}_\mathbbm{t}\otimes H^q_{\mathbbm{t}^\prime}(M)\rightarrow \cdots.
\end{eqnarray*}
Consider the coproduct of such five termed exact sequences corresponding to all pairs $p, q$ with $p+q=n.$ By the naturallity of $\cap$ and the five lemma in the abelian category $\mathcal{H}_{\mathbbm{t}^\prime}$, we can show that $\cap$ is an isomorphism for the term corresponding to $X\otimes M.$
This means that $\mathcal{S}$ is closed under extension and thus $\mathcal{S}$ is a triangulated subcategory of $\mathcal{T}$ containing $\mathcal{H}_\mathbbm{t}$. Hence by \cite[Lemma 3.1]{Big07}, $\mathcal{S}$ coincides with $\mathcal{T}$. The proof is completed.
\end{proof}

\subsection{Triangulated module t-structures}
In this subsection, we aim to introduce triangulated module t-structures and derive their fundamental properties.

Before proceeding further, we recall the definition of a monoidal t-structure. From \cite[Definition 2.10]{XZ25}, a \textit{monoidal t-structure} on a monoidal triangulated category $\mathcal{T}$ is a bounded t-structure $\mathbbm{t} = (\mathcal{T}^{\leq 0}, \mathcal{T}^{\geq 0})$ such that, for some $n \in \mathbb{Z}$,
$\mathcal{T}^{\leq 0} \otimes \mathcal{T}^{\leq n} \subseteq \mathcal{T}^{\leq 0}$ and  $\mathcal{T}^{\geq 0} \otimes \mathcal{T}^{\geq n} \subseteq \mathcal{T}^{\geq 0}.$  The set of integers $n$ with this property is called the \textit{deviation} of $\mathbbm{t}$, denoted by $\operatorname{dev}(\mathbbm{t})$.

Inspired by the above idea, we introduce the definition of a triangulated module t-structure.
\begin{definition}
Let $\mathcal{T}$ be a monoidal triangulated category equipped with a monoidal t-structure $\mathbbm{t}=(\mathcal{T}^{\leq 0}, \mathcal{T}^{\geq 0})$, and let $\mathcal{K}$ be a triangulated module category over $\mathcal{T}$. A bounded t-structure $\mathbbm{t}^\prime=(\mathcal{K}^{\leq 0}, \mathcal{K}^{\geq 0})$ on $\mathcal{K}$ is called a triangulated module t-structure with respect to $\mathbbm{t}$, if there exists an integer
$m\in\mathbb{Z}$ such that
\begin{itemize}
\item[(1)]$\mathcal{T}^{\leq 0}\otimes \mathcal{K}^{\leq m}\subseteq\mathcal{K}^{\leq 0}$;
\item[(2)]$\mathcal{T}^{\geq 0}\otimes \mathcal{K}^{\geq m}\subseteq\mathcal{K}^{\geq 0}$.
\end{itemize}
We call the set of integers $m$ satisfying conditions (1) and (2) the deviation of $\mathbbm{t}^\prime$ with respect to $\mathbbm{t}$, denoted $\operatorname{dev}_{\mathbbm{t}}(\mathbbm{t}^\prime)$.
\end{definition}
\begin{example}\rm\label{ex:standardt-structure}
Let $\mathcal{A}$ be an abelian category. The \textit{standard t-structure} $\mathbbm{t}_{\mathcal{A}}:=(D^{\leq0}_{\mathcal{A}}, D^{\geq 0}_{\mathcal{A}})$ on its derived category $\mathbf{D}^b(\mathcal{A})$ is defined as follows:
$$D^{\leq 0}_\mathcal{A}:=\{X\in \mathbf{D}^b(\mathcal{A})\mid H^i(X)=0, \forall i> 0\},\;\;D^{\geq 0}_\mathcal{A}:=\{X\in \mathbf{D}^b(\mathcal{A})\mid H^i(X)=0, \forall i< 0\}.$$
Let $\mathcal{C}$ be a multiring category and $\mathcal{M}$ a left $\mathcal{C}$-module category. Then the standard t-structure $\mathbbm{t}_{\mathcal{C}}$ is a monoidal t-sturcture on $\mathbf{D}^b(\mathcal{C})$ with $0\in\operatorname{dev}(\mathbbm{t}_{\mathcal{C}})$.  Moreover, $\mathbbm{t}_{\mathcal{M}}$ is a triangulated module t-structure on $\mathbf{D}^b(\mathcal{M})$ with respect to $\mathbbm{t}_{\mathcal{C}}$ and $0\in\operatorname{dev}_{\mathbbm{t}_{\mathcal{C}}}(\mathbbm{t}_{\mathcal{M}})$.
\end{example}

We now present the following lemma.
\begin{lemma}\label{lem:t[-k]}
Let $\mathcal{T}$ be a monoidal triangulated category equipped with a monoidal t-structure $\mathbbm{t}$, and let $\mathcal{K}$ be a triangulated module category over $\mathcal{T}$. Suppose that $\mathbbm{t}^\prime$ is a triangulated module t-structure on $\mathcal{K}$ with respect to $\mathbbm{t}$.
Then for any $k\in\mathbb{Z}$, $\mathbbm{t}^\prime[-k]$ is also a triangulated module t-structure on $\mathcal{K}$ with respect to $\mathbbm{t}$. Moreover, $\operatorname{dev}_{\mathbbm{t}}(\mathbbm{t}^\prime)= \operatorname{dev}_{\mathbbm{t}}(\mathbbm{t}^\prime[-k])$.
\end{lemma}
\begin{proof}
For $m\in \operatorname{dev}_{\mathbbm{t}}(\mathbbm{t}^\prime)$,  we have
$$
\mathcal{T}^{\leq 0}\otimes \mathcal{K}^{\leq m}\subseteq \mathcal{K}^{\leq 0}\;\; \text{and }\mathcal{T}^{\geq 0}\otimes \mathcal{K}^{\geq m}\subseteq\mathcal{K}^{\geq 0}.
$$
The expression above is equivalent to
$$
\mathcal{T}^{\leq 0}\otimes \mathcal{K}^{\leq k+m} \subseteq \mathcal{K}^{\leq k}\;\; \text{and }\mathcal{T}^{\geq  0}\otimes \mathcal{K}^{\geq  k+m}\subseteq\mathcal{K}^{\geq k},
$$
for any $k\in\mathbb{Z},$
which means that $\mathbbm{t}^\prime[-k]$ is a triangulated module t-structure on $\mathcal{K}$ with respect to $\mathbbm{t}$ and $\operatorname{dev}_{\mathbbm{t}}(\mathbbm{t}^\prime)= \operatorname{dev}_{\mathbbm{t}}(\mathbbm{t}^\prime[-k])$.
\end{proof}

Although the deviation of a monoidal t-structure $\mathbbm{t}$ does not need to contain $0$ in general, \cite[Lemma 2.11]{XZ25} shows that one can always choose an integer shift $k \in \mathbb{Z}$ so that $\mathbbm{t}[-k]$ remains a monoidal t-structure and satisfies $0 \in \operatorname{dev}(\mathbbm{t}[-k])$.
Moreover, if $0 \in \operatorname{dev}(\mathbbm{t})$, then $\operatorname{dev}(\mathbbm{t})=\{0\}$ (\cite[Proposition 2.15]{XZ25}), and $\mathcal{H}_{\mathbbm{t}}$ is a multiring category (\cite[Proposition 2.13]{XZ25}), which implies that $\mathbbm{t}$ is compatible. Therefore, in what follows we will always assume $0 \in \operatorname{dev}(\mathbbm{t})$.

For the remainder of this subsection, let $\mathcal{T}$ be a monoidal triangulated category equipped with a monoidal t-structure $\mathbbm{t}=(\mathcal{T}^{\leq 0}, \mathcal{T}^{\geq 0})$ with $0\in\operatorname{dev}(\mathbbm{t})$, and let $\mathcal{K}$ be a triangulated module category over $\mathcal{T}$. Suppose that $\mathbbm{t}^\prime=(\mathcal{K}^{\leq 0}, \mathcal{K}^{\geq 0})$ is a triangulated module t-structure on $\mathcal{K}$ with respect to $\mathbbm{t}$.

In fact, properties analogous to those for monoidal t-structures also hold in the setting of triangulated module t-structures.
\begin{lemma}\label{lem:heartmodule}
If $0\in \operatorname{dev}_{\mathbbm{t}}(\mathbbm{t}^\prime)$, then $\mathcal{H}_{\mathbbm{t}^\prime}$ is a module category over $\mathcal{H}_{\mathbbm{t}}$ whose module product bifunctor is biexact.
\end{lemma}
\begin{proof}
By definition, for any $X\in \mathcal{H}_{\mathbbm{t}}$ and $M\in \mathcal{H}_{\mathbbm{t}^\prime}$, we have $X\otimes M\in \mathcal{K}^{\geq 0}\cap \mathcal{K}^{\leq 0}$. This means that $\mathcal{H}_{\mathbbm{t}}\otimes \mathcal{H}_{\mathbbm{t}^\prime}\subseteq \mathcal{H}_{\mathbbm{t}^\prime}.$ The biexactness of the module product bifunctor is inherited (see \cite[Theorem 1.3.6]{BBD82}).
\end{proof}
As a consequence of \cite[Proposition 2.13]{XZ25} and Lemma \ref{lem:heartmodule}, we have the following corollary.
\begin{corollary}\label{coro:compatible}
If $0\in \operatorname{dev}_{\mathbbm{t}}(\mathbbm{t}^\prime)$, then the module product bifunctor is compatible with $\mathbbm{t}$ and $\mathbbm{t}^\prime.$
\end{corollary}
The following lemma shows that the condition $0 \in \operatorname{dev}_{\mathbbm{t}}(\mathbbm{t}^\prime)$ forces the deviation of $\mathbbm{t}^\prime$ with respect to $\mathbbm{t}$ to be uniquely determined.
\begin{lemma}\label{lem:det=0}
If $0\in \operatorname{dev}_{\mathbbm{t}}(\mathbbm{t}^\prime)$, then $\operatorname{dev}_\mathbbm{t}(\mathbbm{t}^\prime)=\{0\}$.
\end{lemma}
\begin{proof}
Suppose there exists a nonzero $m\in\mathbb{Z}$ such that $m\in \operatorname{dev}_\mathbbm{t}(\mathbbm{t}^\prime)$. Without loss of generality, we assume $m>0.$ Then $\mathcal{T}^{\leq 0}\otimes \mathcal{K}^{\leq m}\subseteq\mathcal{K}^{\leq 0}$.
Note that for any nonzero object $M\in \mathcal{H}_{t^\prime}$, we have $M[-m]\in \mathcal{K}^{\leq m}.$ By \cite[Proposition 2.13]{XZ25} and Lemma \ref{lem:heartmodule},
$\mathcal{H}_{\mathbbm{t}^\prime}$ is a module category over the multiring category $\mathcal{H}_{\mathbbm{t}}$, where $H^0_{\mathbbm{t}}(\mathbbm{1})$ is the unit of $\mathcal{H}_{\mathbbm{t}}$.
It follows that $$H^0_{\mathbbm{t}}(\mathbbm{1})\otimes M[-m]\cong (H^0_{\mathbbm{t}}(\mathbbm{1})\otimes M)[-m] \cong M[-m]\neq 0.$$
However, \cite[Lemma 2.9 (2)]{XZ25} implies that $M[-m]\notin \mathcal{K}^{\leq 0}$, which is a contradiction.
\end{proof}

\begin{proposition}\label{lem:faith=faith}
If $0\in \operatorname{dev}_{\mathbbm{t}}(\mathbbm{t}^\prime)$, then $\mathcal{K}$ is faithful as a $\mathcal{T}$-module category if and only if $\mathcal{H}_{\mathbbm{t}^\prime}$ is faithful as a $\mathcal{H}_{\mathbbm{t}}$-module category.
\end{proposition}
\begin{proof}
The ``only if" implication is immediate; it remains to verify the ``if" direction. For any nonzero objects $T\in\mathcal{T}$ and $K\in\mathcal{K}$, since both $\mathbbm{t}$ and $\mathbbm{t}^\prime$ are bounded, it follows from \cite[Lemma 2.9 (1)]{XZ25} that there exist integers $m$ and $n$ such that $T\in\mathcal{T}^{\leq n}$ with $H^n_{\mathbbm{t}}(T)\neq 0$ and $K\in\mathcal{K}^{\leq m}$ with $H^m_{\mathbbm{t}^\prime}(K)\neq 0$. Combining Proposition \ref{prop:kunneth} and Corollary \ref{coro:compatible}, we know that
$$
H^{n+m}_{\mathbbm{t}^\prime}(T\otimes K)\cong \coprod\limits_{p+q=n+m} H^p_{\mathbbm{t}}(T)\otimes H^q_{\mathbbm{t}^\prime}(K).
$$
Since $H^n_{\mathbbm{t}}(T) \otimes H^m_{\mathbbm{t}^\prime}(K)\neq 0,$ it follows that $T\otimes K\neq 0.$
\end{proof}

Combining Corollary \ref{coro:tensorfaithful} and Proposition \ref{lem:faith=faith}, we have the following corollary.
\begin{corollary}\label{coro:triangulatedfaithful}
Let $\mathcal{C}$ be a multiring category and $\mathcal{M}$ be a left faithful $\mathcal{C}$-module category. Then $\mathbf{D}^b(\mathcal{M})$ is a faithful triangulated module category over $\mathbf{D}^b(\mathcal{C})$. In particular, for a tensor category $\mathcal{C}$ and any $\mathcal{C}$-module category $\mathcal{M}$, the bounded derived category $\mathbf{D}^b(\mathcal{M})$ is faithful over $\mathbf{D}^b(\mathcal{C})$.
\end{corollary}

\section{Triangulated module equivalences and module equivalences}\label{section4}
In this section, we first study the localization of triangulated module categories under certain assumptions.
We then prove our main result: two subcategories of tensor categories are equivalent as module categories if they are both abelian with finitely many simple objects, and their derived categories are equivalent as triangulated module categories via a monoidal triangulated functor.

We first state the following lemma.
\begin{lemma}\emph{(}\cite[Lemma 4.1 and Example 4.5]{CLZ23}\emph{)}\label{lem:finitet=t0}
Let $\mathbbm{t}_1 = (\mathcal{K}_1^{\leq 0}, \mathcal{K}1^{\geq 0})$ be a bounded $t$-structure on a triangulated category $\mathcal{K}$, whose heart $\mathcal{H}_{\mathbbm{t}_1}$ is a locally finite abelian category with finitely many isomorphism classes of simple objects. Then for any bounded t-structure $\mathbbm{t}_2=(\mathcal{K}^{\leq 0}_2, \mathcal{K}^{\geq 0}_2)$ on $\mathcal{K}$, there are integers $m\leq n$ such that $\mathcal{K}^{\leq m}_1\subseteq \mathcal{K}^{\leq 0}_2 \subseteq \mathcal{K}^{\leq n}_1$ and $\mathcal{K}^{\geq n}_1\subseteq \mathcal{K}^{\geq 0}_2 \subseteq \mathcal{K}^{\geq m}_1$.
\end{lemma}

Indeed, a full dense and triangulated functor preserves bounded t-structures.
\begin{lemma}\label{lem:newtstructure}
Suppose that $F:\mathcal{K}_1\rightarrow \mathcal{K}_2$ is a full and dense triangulated functor between triangulated categories $\mathcal{K}_1, \mathcal{K}_2$. Let $\mathbbm{t}=(\mathcal{K}_1^{\leq 0}, \mathcal{K}_1^{\geq 0})$ be a $t$-structure on $\mathcal{K}_1$ and $\mathcal{U}, \mathcal{V}$ be the following full subcategories of $\mathcal{K}_2$:
\begin{eqnarray*}
&&\mathcal{U}:=\{X\in \mathcal{K}_2\mid \text{there exists some } X^\prime\in \mathcal{K}_1^{\leq 0} \text{ such that } F(X^\prime)\cong X\},\\
&&\mathcal{V}:=\{Y\in \mathcal{K}_2\mid \text{there exists some } Y^\prime\in \mathcal{K}_1^{\geq 0} \text{ such that } F(Y^\prime)\cong Y\}.
\end{eqnarray*}
Then $\tilde{\mathbbm{t}}:=(\mathcal{U}, \mathcal{V})$ is a $t$-structure on $\mathcal{K}_2$. In particular, if $\mathbbm{t}$ is bounded, then $\tilde{\mathbbm{t}}$ is also bounded.
\end{lemma}
\begin{proof}
 For any $X\in\mathcal{U}$, there exists some $X^\prime \in \mathcal{K}_1^{\leq 0}$ such that $F(X^\prime)\cong X$. By the fact that $\mathbbm{t}=(\mathcal{K}_1^{\leq 0}, \mathcal{K}_1^{\geq 0})$ is a $t$-structure, we have $X^\prime\in T_1^{\leq 1}$. This means that $F(X^\prime[1])\cong F(X^\prime)[1]\in\mathcal{U}$. It follows that $X\in \mathcal{U}[-1]$ and thus $\mathcal{U}\subseteq \mathcal{U}[-1]$. A similar argument shows that $\mathcal{V}[-1]\subseteq\mathcal{V}.$ For any $Y\in\mathcal{V}$, there exists some $Y^\prime \in \mathcal{K}_1^{\geq 0}$ such that $F(Y^\prime)\cong Y$. Using the fact that $F$ is full, we know that $$\operatorname{Hom}_{\mathcal{K}_2}(X, Y)\cong\operatorname{Hom}_{\mathcal{K}_2}(F(X^\prime), F(Y^\prime))=0.$$ Moreover, there exists some distinguished triangle $$M\rightarrow X^\prime \rightarrow N\rightarrow M[1]$$ in $\mathcal{K}_1$, where $M\in \mathcal{K}_1^{\leq 0}, N\in \mathcal{K}_1^{\geq 1}$. Since $F$ is exact, it follows that $$F(M)\rightarrow F(X^\prime) \rightarrow F(N)\rightarrow F(M)[1]$$ is a distinguished triangle with $F(M)\in\mathcal{U}$ and $F(N)\in\mathcal{V}[-1]$. Therefore, $\tilde{\mathbbm{t}}=(\mathcal{U}, \mathcal{V})$ is a $t$-structure. Since $F$ is full and dense, we know that if $t$ is bounded, then $\tilde{\mathbbm{t}}$ is also bounded.
\end{proof}

A subcategory $\mathcal{A}$ of a multiring category $\mathcal{C}$ is called a \textit{left abelian tensor ideal} if it is an abelian category, the inclusion $i: \mathcal{A} \to \mathcal{C}$ is exact, and it satisfies $X \otimes A \in \mathcal{A}$ for all $X \in \mathcal{C}$ and $A \in \mathcal{A}$. It is clear that $\mathcal{A}$ is a $\mathcal{C}$-module category.
A subcategory $\mathcal{K}$ of a monoidal triangulated category $\mathcal{T}$ is called a \textit{left triangulated tensor ideal} if it is itself an triangulated category with an exact inclusion $i:\mathcal{K}\rightarrow \mathcal{T}$ and satisfies $X\otimes K\in \mathcal{K}$ for any $X\in\mathcal{T}$ and $K\in\mathcal{K}$.

In what follows, let $\mathcal{C}$ and $\mathcal{D}$ be tensor categories, and $\mathcal{A} \subset \mathcal{C}$ and $\mathcal{B} \subset \mathcal{D}$ be left abelian tensor ideals. Assume further that $\mathbf{D}^b(\mathcal{A})$ and $\mathbf{D}^b(\mathcal{B})$ are subcategories of $\mathbf{D}^b(\mathcal{C})$ and $\mathbf{D}^b(\mathcal{D})$, respectively. Then we can show that $\mathbf{D}^b(\mathcal{A})$ and $\mathbf{D}^b(\mathcal{B})$ are left triangulated tensor ideals of $\mathbf{D}^b(\mathcal{C})$ and $\mathbf{D}^b(\mathcal{D})$, respectively. Let $F:\mathbf{D}^b(\mathcal{C})\rightarrow \mathbf{D}^b(\mathcal{D})$ be a monoidal triangulated functor, with $F(\mathbf{D}^b(\mathcal{A}))\subseteq \mathbf{D}^b(\mathcal{B}).$ Thus, via $F$, the category $\mathbf{D}^b(\mathcal{B})$ becomes a left $\mathbf{D}^b(\mathcal{C})$-module category, making $F\mid_{\mathbf{D}^b(\mathcal{A})}: \mathbf{D}^b(\mathcal{A}) \to \mathbf{D}^b(\mathcal{B})$ a triangulated $\mathbf{D}^b(\mathcal{C})$-module functor.

With the help of the preceding lemmas, we now establish the following theorem. This result can, in a sense, be viewed as the converse of Remark \ref{rm:Serre=thick}.
\begin{theorem}\label{thm:thick=serre}
Let $\mathcal{C}$ and $\mathcal{D}$ be tensor categories, and let subcategories $\mathcal{A} \subset \mathcal{C}$ and $\mathcal{B} \subset \mathcal{D}$ be abelian categories with finitely many isomorphism classes of simple objects. Assume further that $\mathbf{D}^b(\mathcal{A})$ and $\mathbf{D}^b(\mathcal{B})$ are left triangulated tensor ideals of $\mathbf{D}^b(\mathcal{C})$ and $\mathbf{D}^b(\mathcal{D})$, respectively. Suppose $F:\mathbf{D}^b(\mathcal{C})\rightarrow \mathbf{D}^b(\mathcal{D})$ is a monoidal triangulated functor such that $F(\mathbf{D}^b(\mathcal{A}))\subseteq \mathbf{D}^b(\mathcal{B})$. If $F\mid_{\mathbf{D}^b(\mathcal{A})}$ is full and dense, then
\begin{itemize}
\item[(1)]$F(\mathcal{A})\subseteq \mathcal{B}$ and $F\mid_{\mathcal{A}}:\mathcal{A}\rightarrow \mathcal{B}$ is a $\mathcal{C}$-module functor.
\item[(2)]$\mathcal{A}/\ker(F\mid_{\mathcal{A}})$ is $\mathcal{C}$-module equivalent to $\mathcal{B}$.
\item[(3)]$\ker(F\mid_{\mathbf{D}^b(\mathcal{A})})=\mathbf{D}^{b}_{\ker(F\mid_{\mathcal{A}})}(\mathcal{A})$, where $\mathbf{D}^{b}_{\ker(F\mid_{\mathcal{A}})}(\mathcal{A})$ is a full subcategory of $\mathbf{D}^{b}(\mathcal{A})$ generated by complexes of which homologies are in $\ker(F\mid_{\mathcal{A}}).$
\end{itemize}
\end{theorem}
\begin{proof}
\begin{itemize}
\item[(1)]
Using Lemma \ref{lem:heartmodule}, we know that $\mathcal{A}$ and $\mathcal{B}$ are left $\mathcal{C}$-module categories.
Let $\mathbbm{t}_{\mathcal{A}}=(D^{\leq 0}_{\mathcal{A}}, D^{\geq 0}_{\mathcal{A}})$, $\mathbbm{t}_{\mathcal{B}}=(D^{\leq 0}_{\mathcal{B}}, D^{\geq 0}_{\mathcal{B}})$ and $\mathbbm{t}_{\mathcal{C}}=(D^{\leq 0}_{\mathcal{C}}, D^{\geq 0}_{\mathcal{C}})$ be the standard t-structures on $\mathbf{D}^{b}(\mathcal{A})$, $\mathbf{D}^{b}(\mathcal{B})$ and $\mathbf{D}^{b}(\mathcal{C})$, respectively. Using Lemma \ref{lem:newtstructure}, we can show that $\tilde{\mathbbm{t}}_{\mathcal{A}}=(\mathcal{U}, \mathcal{V})$ is a bounded t-structure on $\mathbf{D}^{b}(\mathcal{B})$, where
\begin{eqnarray*}
&&\mathcal{U}:=\{U\in \mathbf{D}^{b}(\mathcal{B})\mid \exists U^\prime\in D^{\leq 0}_{\mathcal{A}} \text{ such that } F(U^\prime)\cong U\},\\
&&\mathcal{V}:=\{V\in \mathbf{D}^{b}(\mathcal{B})\mid \exists V^\prime\in D^{\geq 0}_{\mathcal{A}} \text{ such that } F(V^\prime)\cong V\}.
\end{eqnarray*}
Moreover, it is straightforward to show that $\tilde{\mathbbm{t}}_{\mathcal{A}}$ is a triangulated module t-structure on $\mathbf{D}^b(\mathcal{B})$ with respect to $\mathbbm{t}_{\mathcal{C}}$ and $\operatorname{dev}_{\mathbbm{t}_{\mathcal{C}}}(\tilde{\mathbbm{t}}_{\mathcal{A}})=0.$
Then by Lemma \ref{lem:finitet=t0}, there are integers $m\leq n$, which can be chosen maximally and minimally respectively, such that $D^{\leq m}_\mathcal{B}\subseteq \mathcal{U}\subseteq D^{\leq n}_{\mathcal{B}}$ and $D^{\geq n}_\mathcal{B}\subseteq \mathcal{V}\subseteq D^{\geq m}_{\mathcal{B}}$. According to \cite[Lemma 2.9]{XZ25}, there exists some $M^\prime\in D^{\leq 0}_{\mathcal{A}}$ such that $H^n_{\mathbbm{t}_\mathcal{B}}(F(M^\prime))\neq0.$
It follows from Corollary \ref{coro:triangulatedfaithful} that $\mathbf{D}^{b}(\mathcal{B})$ is a faithful $\mathbf{D}^{b}(\mathcal{C})$-module category. Using Proposition \ref{prop:kunneth}, we have $$F(M^\prime)\otimes F(M^\prime)\cong F(M^\prime\otimes M^\prime)\in \mathcal{U},$$ with
$H^{2n}_{\mathbbm{t}_{\mathcal{B}}}(F(M^\prime) \otimes F(M^\prime))\neq 0.$ It follows that $2n\leq n$, which means that $n\leq 0$. A similar argument shows that $m\geq 0$. This implies $m = n = 0$, which forces $\tilde{\mathbbm{t}}_\mathcal{A} = \mathbbm{t}_{\mathcal{B}}$.
 Thus $F(\mathcal{A})\subseteq \mathcal{B}$. According to \cite[Lemma 2.3]{CHZ19}, the restriction $F\mid_{\mathcal{A}}: \mathcal{A} \rightarrow \mathcal{B}$ is exact, hence a $\mathcal{C}$-module functor.
\item[(2)]By (1), $F\mid_{\mathcal{A}}: \mathcal{A} \rightarrow \mathcal{B}$ is an exact $\mathcal{C}$-module functor, it follows that $\ker(F\mid_{\mathcal{A}})$ is a Serre submodule category of $\mathcal{A}$. Since $F$ itself is full and dense, so is its restriction $F\mid_{\mathcal{A}}$. According to \cite[Lemma 3.2]{WZ22}, $\mathcal{A}/\ker(F\mid_{\mathcal{A}})$ is $\mathcal{C}$-module equivalent to $\mathcal{B}$.
\item[(3)]By Remark \ref{rm:Serre=thick}, the proof is completed.
\end{itemize}
\end{proof}

\begin{remark}\rm
It is noted that the above conclusion also holds for multiring categories $\mathcal{C}$ and $\mathcal{D}$ satisfying the condition that for any nonzero object $X\in \mathcal{C}$ (or $\mathcal{D}$), $X \otimes X \neq 0$. Such multiring categories are said to be tensor reduced in \cite[Definition 2.18]{XZ25}.
Substituting
$\mathcal{A}=\mathcal{C}$ and $\mathcal{B}=\mathcal{D}$ into Theorem \ref{thm:thick=serre}, the kernel $\ker(F\mid_{\mathcal{C}})$ is a two-sided Serre tensor ideal of $\mathcal{C}$, which is trivial by \cite[Proposition 3.11]{XZ25}. However, in the tensor reduced setting, we obtain the corresponding localization results for the version of monoidal triangulated categories.
\end{remark}

We proceed to prove the following theorem.
\begin{theorem}\label{thm:eq=}
Let $\mathcal{C}$ and $\mathcal{D}$ be tensor categories, and let subcategories $\mathcal{A} \subset \mathcal{C}$ and $\mathcal{B} \subset \mathcal{D}$ be abelian categories with finitely many isomorphism classes of simple objects. Assume further that $\mathbf{D}^b(\mathcal{A})$ and $\mathbf{D}^b(\mathcal{B})$ are left triangulated tensor ideals of $\mathbf{D}^b(\mathcal{C})$ and $\mathbf{D}^b(\mathcal{D})$, respectively.
Suppose $F: \mathbf{D}^b(\mathcal{C}) \to \mathbf{D}^b(\mathcal{D})$ is a monoidal triangulated functor such that $F(\mathbf{D}^b(\mathcal{A})) \subseteq \mathbf{D}^b(\mathcal{B})$ and that its restriction $F|_{\mathbf{D}^b(\mathcal{A})}: \mathbf{D}^b(\mathcal{A}) \to \mathbf{D}^b(\mathcal{B})$ is an equivalence. Then $\mathcal{A}$ and $\mathcal{B}$ are equivalent as left $\mathcal{C}$-module categories.
\end{theorem}
\begin{proof}
This result is actually a corollary of Theorem \ref{thm:thick=serre}.
\end{proof}
\begin{remark}\rm
Note that in Theorem \ref{thm:eq=}, if $\mathbf{D}^b(\mathcal{A})$ and $\mathbf{D}^b(\mathcal{B})$ are not assumed to be left triangulated tensor ideals but are only assumed to be subcategories which closed under taking tensor products, then the equivalence between $\mathbf{D}^b(\mathcal{A})$ and $\mathbf{D}^b(\mathcal{B})$ induced by $F$ still implies an equivalence between the original categories $\mathcal{A}$ and $\mathcal{B}$, though not necessarily an equivalence of $\mathcal{C}$-module categories.
\end{remark}

In fact, Theorem \ref{thm:eq=} can be seen as a generalization in part of \cite[Theorem 4.5]{XZ25}, which is obtained by setting the subcategories $\mathcal{A}$ and $\mathcal{B}$ to be the entire categories.
For the reader's convenience, we state this conclusion explicitly below.
\begin{corollary}\label{coro:tensorequ}\emph{(}\cite[Theorem 4.5]{XZ25}\emph{)}
Let $\mathcal{C}$ and $\mathcal{D}$ be tensor categories with finitely many isomorphism classes of simple objects. Then $\mathbf{D}^b(\mathcal{C})$ is monoidal triangulated equivalent to $\mathbf{D}^b(\mathcal{D})$ if and only if $\mathcal{C}$ is tensor equivalent to $\mathcal{D}.$
\end{corollary}
It is worth noting that whether the conclusion of Corollary \ref{coro:tensorequ} holds for tensor categories with infinitely many isomorphism classes of simple objects remains unknown. However, Theorem \ref{thm:eq=} does not impose any requirement on whether the two tensor categories have finitely many isomorphism classes of simple objects; it only requires that their respective subcategories have finitely many simple objects. This point can help us investigate the local characterization of monoidal triangulated equivalences between the derived categories of two tensor categories that have infinitely many isomorphism classes of simple objects (such as representation categories of some infinite-dimensional Hopf algebras analogous to quantum groups).

At last, let us give some examples.
\begin{example}\rm
Let $H $ and $H^\prime$ be (possibly infinite-dimensional) Hopf algebras whose coradicals are finite-dimensional. Suppose the derived categories of their finite-dimensional comodule categories are monoidal triangulated equivalent. Corollary \ref{coro:tensorequ} thus yields a tensor equivalence between the finite-dimensional comodule categories $H\text{-}\mathrm{comod}$ and $H^\prime\text{-}\mathrm{comod}$.
If, in addition, their coradicals are Hopf subalgebras, then by \cite[Corollary 5.4]{Far21}, the Grothendieck rings of $H_0\text{-}\mathrm{comod}$ and $H_0^\prime\text{-}\mathrm{comod}$ are isomorphic.
\end{example}

\begin{example}\rm\label{ex:crossed}
Let $\mathcal{C}$ be a tensor category and $A$ is an algebra in $\mathcal{C}$ (see \cite[Definition 7.8.1]{EGNO15}). Denote $\mathrm{Mod}_{\mathcal{C}}(A)$ by the categories of right modules over $A$ in $\mathcal{C}$ (see \cite[Definition 7.8.5]{EGNO15}).
According to \cite[Exercises 7.8.7 and 7.8.8]{EGNO15}, $\mathrm{Mod}_{\mathcal{C}}(A)$ is a left abelian tensor ideal of $\mathcal{C}$. Assume further that $\mathbf{D}^b(\mathrm{Mod}_{\mathcal{C}}(A))$ is a subcategory of $\mathbf{D}^b(\mathcal{C})$, then it is in fact a left triangulated tensor ideal.
In particular, let $H$ be a Hopf algebra and $A$ be a right $H$-module algebra. This means $A$ is equipped with a right $H$-module structure such that both its multiplication and unit maps are morphisms of $H$-modules. Then the smash product algebra $H\# A$, defined as in \cite[Definition 4.1.3]{Mon93}, becomes an algebra in $H\text{-}\mathrm{mod}$. According to \cite[Exercise 7.8.32]{EGNO15}, the module category of $A$ in $H\text{-}\mathrm{mod}$ is equivalent to the category of $H\# A$-modules. It should be remarked that for any algebra $B$, we have $B=\k 1\#B$. In this case, however, Theorem \ref{thm:eq=} becomes trivial.

If in addition, suppose $H$ and its dual $H^*$ are finite-dimensional semisimple Hopf algebras over $\k$, and let $A, B$ be finite-dimensional right $H$-module algebras. If the derived categories of $H\#A$ and $H\#B$ are derived equivalent, then according to \cite[Theorem 3.8]{LS13}, the algebras $A$ and $B$ and the two corresponding smash product algebras have the same derived representation type. Moreover, if $\mathbf{D}^b(H\#A\text{-}\mathrm{mod})$ and $\mathbf{D}^b(H\#B\text{-}\mathrm{mod})$ are subcategories of $\mathbf{D}^b(H\text{-}\mathrm{mod})$,
and such a derived equivalence is also a triangulated $\mathbf{D}^b(H\text{-}\mathrm{mod})$-module functor induced by the monoidal triangulated endofunctor of $\mathbf{D}^b(H\text{-}\mathrm{mod})$, then by Theorem \ref{thm:eq=}, the module categories of $H\#A$ and $H\#B$ are equivalent as module categories over $H\text{-}\mathrm{mod}$. Consequently, it follows from \cite[Theorem 2.6]{LZ07} that the algebras $A$ and $B$ and the two corresponding smash product algebras also have the same representation type.
\end{example}

\section*{Acknowledgement}
The author would like to thank Professors Xiao-Wu Chen, Gongxiang Liu and Zhengfang Wang for their valuable comments and suggestions.
The author also thanks Menggao Li for many informative discussions on this subject.

\end{document}